%% file: thins_nonarithmetic.tex
\documentclass{amsart}
\usepackage{standalone}
\usepackage{graphicx}
\usepackage{amssymb, amsfonts, amsmath, amsthm}
\usepackage[all]{xy}
\usepackage{amsmath}
\usepackage{verbatim}
\usepackage[switch, modulo]{lineno}
\usepackage{hyperref}
\usepackage{tikz-cd}
\usepackage{enumerate}
\graphicspath{{figures/}}

\newcommand{\abs}[1]{\left| #1 \right|}

\newcommand{\D}{\mathcal{D}}
\newcommand{\R}{\mathbb{R}}

\newcommand{\tr}{{\rm tr}}

\newcommand{\C}{\mathbb{C}}

\newcommand{\Q}{\mathbb{Q}}

\newcommand{\GL}{\operatorname{GL}}
\newcommand{\SL}{\operatorname{SL}}

\newcommand{\SO}{\operatorname{SO}}
\newcommand{\SU}{\operatorname{SU}}

\newcommand{\CC}{\mathcal{C}}

\newcommand{\Z}{\mathbb{Z}}

\newcommand{\NN}{\mathbb{N}}
\newcommand{\HH}{\mathbb{H}}

\newcommand{\bs}{\backslash}
\newcommand{\Diag}{\text{Diag}}

\newcommand{\Ad}{\mathsf{Ad}}

\renewcommand{\O}{\mathcal{O}}

\newtheorem{theorem}{Theorem}

\newtheorem{proposition}[theorem]{Proposition}
\newtheorem{corollary}[theorem]{Corollary}
\newtheorem{lemma}[theorem]{Lemma}

\numberwithin{equation}{section}
\numberwithin{theorem}{section}

\newcommand{\Isom}{{\rm Isom}}

\begin{document}

\title[]{Thin subgroups isomorphic to Gromov--Piatetski-Shapiro lattices}
\author{Samuel A.\ Ballas 
}
\email{ballas@math.fsu.edu}
\date{\today}
\address{Department of Mathematics\\ 
Florida State University\\ Tallahassee, FL 32306, USA}

\maketitle

\begin{abstract}
	In this paper we produce many examples of thin subgroups of special linear groups that are isomorphic to the fundamental group non-arithmetic hyperbolic manifolds. Specifically, we show that the non-arithmetic lattices in $\SO(n,1,\R)$ constructed by Gromov and Piatetski-Shapiro can be embedded into $\SL(n+1,\R)$ so that their images are thin subgroups. 
	\end{abstract}

\tableofcontents

\section*{Introduction}

Let $G$ be a semisimple Lie group and let $\Gamma$ be a finitely generated subgroup. We say that $\Gamma$ is a \emph{thin subgroup of G} if there is a lattice $\Lambda\subset G$ containing $\Gamma$ such that 
\begin{itemize}
	\item $\Gamma$ has infinite index in $\Lambda$
	\item $\Gamma$ is Zariski dense in $G$. 
\end{itemize}
Intuitively, such groups are very sparse in the sense that they have infinite index in a lattice, but at the same time are dense in an algebraic sense. Note, that if one relaxes the first condition above, then $\Gamma$ would be a lattice, so another way of thinking of thin groups is as infinite index analogs of lattices in semisimple Lie groups. 

Over the last several years, thin groups have been the subject of much research, much of which has been motivated by the observation that many theorems and conjectures in number theory can be phrased in terms of counting primes in orbits of groups that are ``abelian analogs of thin groups.'' Here are two examples. First, let $G=\R$, $b,m\in \NN$ such that $(b,m)=1$, $\Delta=\Z$ and $\Gamma=m\Z$. The orbit $b+\Gamma$ is an arithmetic progression and Dirichlet's theorem on primes in arithmetic progressions is equivalent to this orbit containing infinitely many primes. Next, let $G=\R^2$, $\Delta=\Z^2$, $\Gamma=\langle (1,1)\rangle$ and $b=(1,3)\in \Z^2$. The orbit $b+\Gamma=\{(m,m+2)\mid m\in \Z\}$ and the twin prime conjecture is equivalent to the statement that this orbit contains infinitely many points whose components are both prime. Note that in the first case $\Gamma$ is a lattice in $G$, but in the second case $\Gamma$ has infinite index in $\Delta$ and is an analog of a thin group (sans Zariski density) in $G$.

 This orbital perspective was used by Brun to attack the twin primes conjecture using ``combinatorial sieving'' techniques. Although the full conjecture remains unproven these techniques did yield some powerful results. For instance, using these methods, Chen \cite{ChenTwinPrime} was able to prove that there are infinitely many pairs $n$ and $n+2$ such that one is prime and the other is the product of at most 2 primes. More details of this perspective are explained in the excellent surveys of Bourgain \cite{BourgainExpanderSurvey} and Lubotzky \cite{LubotExpanderSurvey}.

 Inspired by these results, Bourgain, Gamburd, and Sarnak \cite{BourGambSarnAffineSeive} a developed complementary ``affine sieving'' techniques to analyze thin group orbits. In this context, the thinness property of the group gives enough control of orbits to execute these counting arguments. Again, much of this is described in Lubotzky's survey \cite{LubotExpanderSurvey}. 

Given these connections it is desirable to produce examples of thin groups and understand what types of groups are thin. Presently, there are many constructions of thin groups. For instance, in recent work of Fuchs and Rivin \cite{FuchsRivinThin} it is shown that if one ``randomly'' selects two matrices in $\SL(n,\Z)$ then with high probability, the group they generate is a thin subgroup of $\SL(n,\R)$. However, the groups constructed in this way are almost always free groups. There are also several constructions that allow one to produce thin subgroups isomorphic to fundamental groups of closed surfaces in a variety of algebraic groups (see \cite{CooperFuterThin,KahnLabourieMozes,KahnMark,KahnWright}, for instance). Given these examples one may ask which isomorphism classes of groups are thin? More precisely, if $G$ is a semisimple Lie group and $H$ is an abstract finitely generated group then we say that $H$ can be \emph{realized as a thin subgroup of $G$} if there is in embedding $\iota:H\to G$ whose image is a thin subgroup of $G$. With this definition in hand we can rephrase the previous question as: given a semisimple algebraic group $G$, what isomorphism types of groups can be realized as thin subgroups of $G$? 
Recent work of the author and D.\ Long \cite{BLThinBend} shows that there are many additional isomorphism types of groups that can arise as thin subgroups of special linear groups. More precisely, in \cite{BLThinBend} it is shown that fundamental groups of arithmetic hyperbolic $n$-manifolds of ``orthogonal type'' can be realized as thin subgroups. In the present work, we extend the techniques of \cite{BLThinBend} to produce infinitely many examples of non-arithmetic hyperbolic $n$-manifolds whose fundamental groups can be realized as thin subgroups of $\SL_{n+1}(\R)$. Our main result is:

\begin{theorem}\label{thm:mainthm}
	For each $n\geq 3$, there is an infinite collection $\mathcal{C}_n$ of non-arithmetic hyperbolic $n$-manifolds with the property that if $M^n\in \mathcal{C}_n$ then $\pi_1(M)$ can be realized as a thin subgroup of $\SL_{n+1}(\R)$.  Furthermore, the collection $\CC_n$ contains representatives from infinitely many commensurability classes of both compact and non-compact manifolds. 
	
\end{theorem}

It should be noted that the collection $\CC_n$ appearing in Theorem \ref{thm:mainthm} can be described fairly explicitly, and roughly speaking consists of the hyperbolic manifolds coming from the non-arithmetic lattices in $\SO(n,1,\R)$ constructed by Gromov--Piatetski-Shapiro in \cite{GPS}. 

\subsection*{Outline of paper}
In Section \ref{sec:GPS} we recall the Gromov--Piatetski-Shapiro construction of non-arithmetic lattices in $\SO(n,1,\R)$ and define the collection $\CC_n$ appearing in Theorem \ref{thm:mainthm}. In Section \ref{sec:lattices} we show that the fundamental group of any element of $\CC_n$ can be embedded in several lattices in $\SL_{n+1}(\R)$. Finally, in Section \ref{sec:thinness} we prove Theorem \ref{thm:mainthm} by showing that the images of the previously mentioned embeddings are thin subgroups. 

\subsection*{Acknowledgements} The author would like to thank Darren Long for several helpful conversations during the preparation of this work and Matt Stover for providing references that greatly simplified the proof of Lemma \ref{lem:adj_tr_field}. The author was partially supported by the NSF grant DMS-1709097. 


\section{Gromov--Piatetski-Shapiro lattices}\label{sec:GPS}

In their 1987 paper \cite{GPS}, Gromov--Piatetski-Shapiro describe a method for constructing infinitely many non-arithmetic lattices in $\SO(n,1,\R)$. In this section we describe their construction and the construction of the lattices appearing in Theorem \ref{thm:mainthm}. 

Let $K$ be a totally real number field of degree $d+1$ with ring of integers $\O_K$. There are $d+1$ embeddings $\{\sigma_0,\ldots,\sigma_{d}\}$ of $K$ into $\R$. Using the embedding $\sigma_0$ we will implicitly regard $K$ as a subset of $\R$. In this way, it makes sense to say that elements of $F$ are positive (resp.\ negative). Let $s_K:K^\times\to \Z_{\geq 0}$  where $s_K(a)=\abs{\{i\geq 1\mid \sigma_i(a)>0\}}$. In other words, $s_K(a)$ counts the non-identity embeddings for which $a$ has positive image.     

Next, let $\alpha,\beta,a_2\ldots a_{n+1}\in \O_K$ be positive elements such that 
\begin{itemize}
	\item $\beta/\alpha$ is not a square in $K$
	\item $s_K(\alpha)=s_K(\beta)=s_K(\alpha_i)=d$ for $1\leq i\leq n$
	\item $s_K(a_{n+1})=0$
\end{itemize}
  Next, define quadratic forms 
\begin{align}\label{eq:quad_forms}
	J_1=\alpha x_1^2+\sum_{i=2}^{n}a_ix_i^2-a_{n+1}x_{n+1}^2\\
	\nonumber
	J_2=\beta x_1^2+\sum_{i=2}^{n}a_ix_i^2-a_{n+1}x_{n+1}^2
\end{align}

If $A\subset \R$ is a subring containing 1 then we define 
$$\SO(J_i,A)=\{B\in \SL_{n+1}(A)\mid J_i(Bv)=J_i(v)\ \forall v\in \R^{n+1}\}.$$
 \noindent Using this notation define $\Gamma_1=\SO(J_1,\O_K)$ and $\Gamma_2=h\SO(J_2,\O_K)h^{-1}$, where $h=\Diag(\sqrt{\beta/\alpha},\ldots,1)$. Note that both $\Gamma_1$ and $\Gamma_2$ are lattices in $\SO(J_1,\R)$, however, since $\beta/\alpha$ is not a square in $K$ it follows from \cite{GPS} (see Cor 2.7 and \S 2.9) that these lattices are not commensurable. 
 
 There is a model for hyperbolic $n$-space given by 
 $$\HH^n=\{v\in \R^{n+1}\mid J_1(v)=-1,\ v_{n+1}>0 \}.$$ The identity component $\SO(J_1,\R)^\circ$ of $\SO(J_1,\R)$ consists of the orientation preserving isometries of $\HH^n$ (see \cite[\S 3.2]{Rat} for details). By passing to finite index subgroups we can assume that $\Gamma_i\subset \SO(J_1,\R)^\circ$, and so $\HH^n/\Gamma_i$ is a finite volume hyperbolic orbifold for $i=1,2$. 
  
  The lattice $\Gamma_2\subset \SO(J_1,L)$, where $L=K(\sqrt{\beta/\alpha})$. Note that because $\alpha$ and $\beta$ are positive and $s_K(\alpha)=s_K(\beta)=d$ it follows that $L$ is also totally real. Furthermore, for every $\gamma\in \Gamma_2$, $\tr(\gamma)\in \O_K\subset \O_L$. The following lemma then shows that by passing to a subgroup of finite index we may assume that $\Gamma_2\subset \SO(J_1,\O_L)$. This result seems well known to experts, but we include a proof for the sake of completeness.
  \begin{lemma}\label{lem:int_trace}
  	Let $k\subset \C$ be a number field and let $\O_k$ be the ring of integers of $k$. If $\Gamma\subset \GL_{n}(k)$ acts irreducibly on $\C^n$ and has the property that $\tr(\gamma)\in \O_k$ for each $\gamma\in \Gamma$ then there is a finite index subgroup $\Gamma'\subset \Gamma$ such that $\Gamma'\subset \GL_n(\O_k)$. 
  \end{lemma}

  \begin{proof}
  If $A\subset k$ is a subring then let $A\Gamma=\{\sum_ia_i\gamma_i\mid a_i\in A, \gamma_i\in \Gamma\}$. Note that in this definition all sums have finitely many terms. By \cite[Prop 2.2]{Bass}, $\O_k\Gamma$ is an order in the central simple algebra $k\Gamma$. The order $\O_k\Gamma$ is contained in some maximal order $\D$ in $M_n(k)$ ($n\times n$ matrices over $k$). Let $\D^1\subset \SL_n(k)$ be the norm 1 elements of $\D$. Then $M_n(\O_k)$ is also an order in $M_n(k)$ whose group of norm 1 elements is $\SL_n(\O_k)$. It is a standard result using restriction of scalars that groups of norm 1 elements in maximal orders of $M_n(k)$ are commensurable. Roughly speaking this is a consequence of the fact that the intersection of two orders is again an order and the unit groups of these orders are irreducible lattices in $\SL_n(\R)\times \SL_n(\R)$ (see \cite[\S 5.1 and Ex 5.1 \#7]{WitteMorris}). It follows that $\D^1\cap \SL_n(\O_k)$ has finite index in $\D^1$ and so $\Gamma\cap \SL_n(\O_k)$ has finite index in $\Gamma$. 
  \end{proof}
  
  \noindent Note that since $\Gamma_2$ is a lattice in $\SO(J_1,\O_K)$ it acts irreducibly on $\C^{n+1}$, and so by applying Lemma \ref{lem:int_trace} we may assume that $\Gamma_2\subset \SO(J_1,\O_L)$.

  Denote by $\SO(n-1,1,\R)$ the subgroup of $\SO(J_1,\R)$ that preserves both complementary components in $\R^{n+1}$ of the hyperplane $P$ given by the equation $x_1=0$. The intersection $P\cap \HH^n$ is a model for hyperbolic $(n-1)$-space, $\HH^{n-1}$ and the group $\SO(n-1,1,\R)$ can be identified with the subgroup of orientation preserving isometries of $\HH^{n-1}$. Next, let $\hat\Gamma=\Gamma_1\cap\Gamma_2\cap \SO(n-1,1,\R)$. Since each $\Gamma_i\cap \SO(n-1,1,\R)$ is sublattice of the lattice $\SO(n-1,1, \O_L)$ in $\SO(n-1,1,\R)$, it follows that $\hat \Gamma$ is also a lattice in $\SO(n-1,1,\R)$. It follows that so $\HH^{n-1}/\hat\Gamma$ is a hyperbolic $(n-1)$ orbifold. By passing to finite index subgroups we may arrange the following properties:

  \begin{enumerate}[1.]
  	\item $\Gamma_i$ is torsion-free and contained in the identity component of $\SO(J_1,\R)$. This component is isomorphic to $\Isom^+(\HH^n)$, and so $M_i:=\HH^n/\Gamma_i$ is a finite volume hyperbolic manifold (apply Selberg's Lemma and the fact that $\SO(J,\R)^\circ$ has finite index in $\SO(J,\R)$).
  	
  	\item Since $\Sigma=\HH^{n-1}/\hat\Gamma$ is a totally geodesic we may assume that $\Sigma$ is a hyperbolic $(n-1)$-manifold and this manifold is embedded in both $M_1$ and $M_2$ (see \cite[Thm 1]{Bergeron}).
  	\item If $M_i$ is non-compact then all cusps of $M_i$ are diffeomorphic to an $(n-1)$-torus times an interval (apply \cite[Thm 3.1]{MRSCollisions})
  	\item The complement $\hat M_i=M_i\bs \Sigma$ is connected for $i=1,2$ (see \cite[Thm 2]{Bergeron}). 
  \end{enumerate}

  The manifold $\hat M_i$ is a convex submanifold of $M_i$ and so $\hat M_i=V_i/\hat\Gamma_i$, where $V_i$ is a component of the preimage of $\hat M_i$ in $\HH^n$ under the universal covering projection $\HH^n\to \HH^n/\Gamma_i=M_i$, and $\hat \Gamma_i$ is a subgroup of $\Gamma_i$ that stabilizes $V_i$. The manifold $\hat M_i$ is a hyperbolic manifold with totally geodesic boundary equal to two isometric copies of $\Sigma$, and so it is possible to glue $\hat M_1$ and $\hat M_1$ along $\Sigma$ to form the finite volume hyperbolic manifold $N$ (see \cite[\S 6.5]{WitteMorris} for details). The manifold $N$ can be realized as $\HH^n/\Delta$ where, after appropriately conjugating $\hat \Gamma_i$ in $\Gamma_i$, we may assume that 
  
  \begin{align}\label{eq:Delta_Presentation}
  	\Delta=\langle \hat\Gamma_1,\hat\Gamma_2,s\rangle. 
  \end{align}
Here $s$ comes from a ``graph of spaces'' description of $N$ and can thus be written as a product $s=s_2s_1$, where $s_i$ is the isometry corresponding to an appropriate lift to $V_i$ a curve in $M_i$ whose algebraic intersection with $\Sigma$ is 1 (See Figure \ref{fig:HNN}). In \cite[\S 2.9]{GPS} it is shown that $\Delta$ is a non-arithmetic lattice in $\SO(J_1,\R)$. If $N=\HH^n/\Delta$ then we call $N$ an \emph{interbreeding of $M_1$ and $M_2$}.  

Since $\Gamma_1,\Gamma_2\subset \SO(J_1,\O_L)$ it follows that $\Delta\subset \SO(J_1,\O_L)$. As a result, we call the field $L$ the \emph{field of definition of $\Delta$}. Let $\CC_n$ be the collection of hyperbolic $n$-manifolds coming from the above interbreeding construction.
  
  \begin{figure}
  \centering
\def \svgwidth{.38 \textwidth}
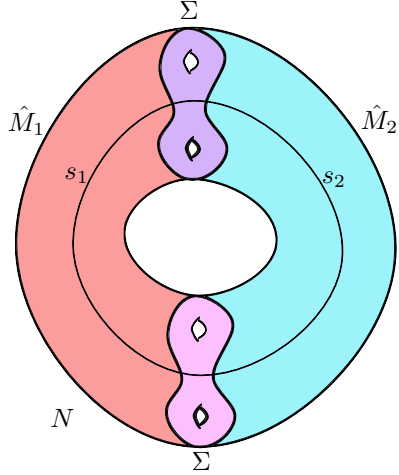
\caption{An graph of spaces description of the manifold $N$.\label{fig:HNN}}	
  \end{figure}

  We close this section by proving the following result:
  
  \begin{proposition}\label{prop:inf_comm_classes}
  	The collection $\CC_n$ contains representatives of infinitely many commensurability classes of  both closed and non-compact hyperbolic $n$-manifolds satisfying the properties 1-4 from above.
  \end{proposition}

   To prove this we will need the following invariant, originally due to Vinberg \cite{VinbergFieldofDef}. Let $\Gamma$ be a Zariski dense subgroup of a Lie group $H$ with Lie algebra $\mathfrak{h}$. The adjoint action of $\Gamma$ on $\mathfrak{h}$ gives a representation $\Ad:\Gamma\to \mathfrak{gl(h)}$. In \cite{VinbergFieldofDef} it is shown that the field $\Q(\{\tr(\Ad(\gamma))\mid \gamma\in \Gamma\})$ is an invariant of the commensurability class of $\Gamma$ in $H$. This field is called the \emph{adjoint trace field of $\Gamma$}.   
  
  Next, let $N=\HH^n/\Delta\in \CC_n$, then $\Delta$ is a lattice in $\SO(J_1,\R)$, which is Zariski dense by the Borel density theorem. The following lemma allows us to compute the adjoint trace field of $\Delta$. It is an immediate corollary of a theorem of Mila (see \cite[Thm 4.7]{MilaThesis}) once it is observed that $L$ is the smallest extension of $K$ over which the forms $J_1$ and $J_2$ are isometric. 
  
  \begin{lemma}\label{lem:adj_tr_field}
  	Let $N=\HH^n/\Delta\in \CC_n$ and let $L$ be the field of definition of $\Delta$, then $L$ is the adjoint trace field of $\Delta$. 
  \end{lemma}

\begin{proof}[Proof of Proposition \ref{prop:inf_comm_classes}]
From \cite{GPS}, it follows that $N=\HH^n/\Delta$ is compact if and only if the field $K$ used to construct $\Delta$ is not equal to $\Q$. For each choice of a totally real field $K$ and a pair $\alpha,\beta\in K$ so that $\frac{\alpha}{\beta}$ is not a square in $K$ we can produce an element $N\in\CC_n$ via the interbreeding construction.  By varying the choices of $\alpha$ and $\beta$ we can produce infinitely many distinct $L=K(\sqrt{\beta/\alpha})$ for each choice of $K$. It follows from Lemma \ref{lem:adj_tr_field} the the corresponding $N$ are representatives of infinitely many commensurability classes of both compact and non-compact hyperbolic $n$-manfiolds. 
\end{proof}

  \section{Lattices in $\SL_{n+1}(\R)$}\label{sec:lattices}
  
  In this section we describe the lattices $\Delta\subset \SL_{n+1}(\R)$ in which our thin groups will ultimately live.  Let $J_1$ be one of the forms constructed in Section \ref{sec:GPS} and let $L$ be the corresponding (totally real) field of definition. Let $M=L(\sqrt{r})$, where $r\in L$ is positive, square-free, and $s_L(r)=0$. The number field $M$ is a quadratic extension of $L$ and we let $\tau:M\to M$ be the unique non-trivial Galois automorphism of $M$ over $L$. In this context, we can extend the quadratic form $J_1$ on $L^{n+1}$ to a ``Hermitian'' form on $M^{n+1}$. Let $N_{M/L}:M\to L$ given by $N_{M/L}(x)=x\tau(x)$ be the norm of the field extension $M/L$. Next let $x=(x_1,\ldots,x_{n+1})\in M^{n+1}$ and define $H_1:M^{n+1}\to L$ as 
   
  $$H_1(x)=\alpha N_{M/L}(x_1)+\sum_{i=1}^na_iN_{M/L}(x_i)-a_{n+1}N_{M/L}(x_{n+1})$$
  
  Note that this defines a Hermitian form in the sense that if $x\in M^{n+1}$ and $\lambda\in M$ then $H_1(\lambda x)=N_{M/L}(\lambda)H_1(x)$. Furthermore, since $L$ is the fixed field of $\tau$ it follows that $H_1$ reduces to $J_1$ when restricted to $L^{n+1}$. 
  
   Next, we can define a unitary analogue of $\SO(J_1,\O_M)$ as 
  $$\SU(J_1,\tau,\O_M)=\{A\in \SL_{n+1}(\O_M)\mid H_1(Av)=H_1(v)\ \forall v\in M^{n+1} \}.$$
  
 \noindent It is well known (see \cite[\S 6.8]{WitteMorris},  for example) that $\SU(J_1,\tau,\O_M)$ is an arithmetic lattice in $\SL_{n+1}(\R)$.

 Let $N=\HH^n/\Delta$ be one of the manifolds from $\CC_n$. By construction, the manifold $N$ contains the embedded totally geodesic hypersurface $\Sigma=\HH^{n-1}/\hat \Gamma$, and so it is possible to deform $\Delta$ inside of $\SL_{n+1}(\R)$ using the bending construction of Johnson and Millson (see \cite{JohnMill}).
 
 Specifically, let $c_t=\Diag(e^{-nt},e^t,\ldots,e^t)\in \SL_{n+1}(\R)$. It is easy to check that $c_t$ centralizes $\SO(n-1,1,\R)$. Since $\Sigma$ is assumed to be non-separating, we see that write $\Delta$ as an HNN extension $\Delta\cong \hat\Delta\ast_s$, where $\hat\Delta$ is isomorphic to the fundamental group of $N\backslash \Sigma$ and $s$ is a free letter. In this context, we may view $\hat\Delta\subset \SO(J_1,\O_L)$ and $s\in \SO(J_1,\O_L)$ and observe that as a subgroup of $\SO(J_1,\O_L)$ we can write $\Delta=\langle \hat \Delta,s\rangle$. We now define a new family of subgroups $\Delta_t=\langle \Delta,c_ts\rangle\subset \SL_{n+1}(\R)$. Using basic theory of HNN extensions, it is easy to see that, since $c_t$ centralizes the fundamental group of $\Sigma$, as an abstract group $\Delta_t$ is a quotient of $\Delta$. However, by using the following result due to Benoist  in the compact case \cite{BenoistCDIII} and Marquis in the non-compact case \cite{MarquisBending}, we can actually say much more

 \begin{proposition}\label{prop:allisom}
 	For each $t$, the group $\Delta_t$ is isomorphic to $\Delta$. 
 \end{proposition}

 Next, we show for certain values of $t$  the group $\Delta_t$ is contained in one of the unitary lattices constructed above. Specifically, if $N=\HH^n/\Delta$ is contained in $\CC_n$, let $J_1$ and $L$ be such that $\Delta\subset \SO(J_1,\O_L)$. Recall that the field $L$ is totally real of degree $d+1$ over $\Q$ and so there are $d+1$ embeddings $\{\sigma_0=Id,\ldots,\sigma_d\}$ of $L$ into $\R$.  We can use Lemma 3.1 of \cite{BLThinBend} to produce a unit $u\in \O_L^\times$ with the property that $\abs{u}>2$  and $0<\abs{\sigma_i(u)}<1$ for $1\leq i\leq d$. Let $p(x)=x^2-ux+1$ and let $M=L(v)$, where $v$ is one of the roots of $p(x)$. It is easy to check that the discriminant of $p(x)$ is $u^2-4$ and so $M=L(\sqrt{u^2-4})$. By construction $s_L(u^2-4)=0$, and so $\SU(J_1,\tau,\O_M)$ is an arithmetic lattice in $\SL_{n+1}(\R)$, where $\tau:M\to M$ is the non-trivial Galois automorphism of $M$ over $L$. The next lemma says that by carefully choosing $t$, we can arrange that $\Delta_t\subset \SU(J_1,\tau,\O_M)$. 
 
 \begin{lemma}\label{lem:deltunitary}
 	Let $u$ be as above, then if $t=\log(u)$ then $\Delta_t\subset \SU(J_1,\tau,\O_M)$.
 \end{lemma}
 
 This is basically Lemma 3.4 of \cite{BLThinBend}, but the proof is short so we include it here for the sake of completeness.
 
 \begin{proof}
 	Recall from above that there is a subgroup $\hat\Delta\subset \SO(J_1,\O_L)$ and $s\in \SO(J_1,\O_L)$ so that  $\Delta=\langle \hat \Delta, s\rangle$ and $\Delta_t=\langle \hat \Delta,c_t s\rangle$, where $c_t=\Diag(e^{-nt},e^t,\ldots,e^t)\in \SL_{n+1}(\R)$. Since $\SO(J_1,\O_L)\subset \SU(J_1,\tau,\O_N)$ the proof will be complete if we can show that $c_t\in \SU(J_1,\tau,\O_M)$. 
 	
 	If $t=\log(u)$ then $c_t=\Diag(u^{-n},u,\ldots,u)$. Furthermore, since $\tau(u)$ is the other root of $p(x)$ it follows that $u\tau(u)=1$, or in other words $\tau(u)=u^{-1}$. It follows that $c_t^\ast=\Diag(u^n,u^{-1},\ldots u^{-1})$. A simple computation then shows that for each $v\in M^{n+1}$, $H_1(c_tv)=H_1(v)$, and so $c_t\in \SU(J_1,\tau,\O_M)$. 
 \end{proof}
 
 By combining Lemma \ref{lem:deltunitary} and Proposiiton \ref{prop:allisom} we get the following corollary
 
 \begin{corollary}\label{cor:candgroup}
 	For each $N=\HH^n/\Delta\in \CC_n$ there are infinitely many lattices $\Lambda\subset \SL_{n+1}(\R)$ that contains a subgroup $\Delta'$ isomorphic to $\Delta$. 
 \end{corollary}

 \section{Certifying thinness}\label{sec:thinness}
  
  The main goal of this section is to complete the proof of Theorem \ref{thm:mainthm}. The proof consist of proving that the subgroups constructed in the previous section are thin.

  \begin{proof}[Proof of Theorem \ref{thm:mainthm}]
  	Recall, that if $N=\HH^n/\Delta\in \CC_n$ from Corollary \ref{cor:candgroup} it follows that we can find a lattice $\Lambda\subset \SL_{n+1}(\R)$ and a subgroup $\Delta'\subset \Lambda$ that is isomorphic to $\Delta$. 
  	
  	Since $\Delta'$ was obtained from $\Delta$ via a bending construction it follows from \cite[Prop 4.1]{BLThinBend}  that $\Delta'$ is Zariski dense in $\SL_{n+1}(\R)$. The proof will be complete if we can show that $\Delta'$ has infinite index in $\Lambda$. Suppose for contradiction that this index is finite. Since $\Lambda$ is a lattice in $\SL_{n+1}(\R)$ this implies that $\Delta'$ is also a lattice in $\SL_{n+1}(\R)$. However, $\Delta'$ is isomorphic to $\Delta$ and $\Delta$ is a lattice in the Lie group $\SO(n,1)^\circ$. However, $\SO(n,1)^\circ$ and $\SL_{n+1}(\R)$ are not isomorphic and so this contradicts the Mostow rigidity theorem (see \cite[Thm 15.1.2]{WitteMorris}).
  \end{proof}   
  
  
  	
  

\bibliographystyle{plain}
\bibliography{bibliography}

\end{document}

%% file: figures/HNN.pdf_tex
\begingroup%
  \makeatletter%
  \providecommand\color[2][]{%
    \errmessage{(Inkscape) Color is used for the text in Inkscape, but the package 'color.sty' is not loaded}%
    \renewcommand\color[2][]{}%
  }%
  \providecommand\transparent[1]{%
    \errmessage{(Inkscape) Transparency is used (non-zero) for the text in Inkscape, but the package 'transparent.sty' is not loaded}%
    \renewcommand\transparent[1]{}%
  }%
  \providecommand\rotatebox[2]{#2}%
  \ifx\svgwidth\undefined%
    \setlength{\unitlength}{578.16948838bp}%
    \ifx\svgscale\undefined%
      \relax%
    \else%
      \setlength{\unitlength}{\unitlength * \real{\svgscale}}%
    \fi%
  \else%
    \setlength{\unitlength}{\svgwidth}%
  \fi%
  \global\let\svgwidth\undefined%
  \global\let\svgscale\undefined%
  \makeatother%
  \begin{picture}(1,1.00960749)%
    \put(0,0){\includegraphics[width=\unitlength,page=1]{HNN.pdf}}%
    \put(0.10508291,0.62930715){\color[rgb]{0,0,0}\makebox(0,0)[lb]{\smash{$s_1$}}}%
    \put(0.65019573,0.62092084){\color[rgb]{0,0,0}\makebox(0,0)[lb]{\smash{$s_2$}}}%
    \put(0.37484386,0.01151258){\color[rgb]{0,0,0}\makebox(0,0)[lb]{\smash{$\Sigma$}}}%
    \put(0.34828708,0.96755663){\color[rgb]{0,0,0}\makebox(0,0)[lb]{\smash{$\Sigma$}}}%
    \put(-0.01512145,0.73134106){\color[rgb]{0,0,0}\makebox(0,0)[lb]{\smash{$\hat M_1$}}}%
    \put(0.73266151,0.73413654){\color[rgb]{0,0,0}\makebox(0,0)[lb]{\smash{$\hat M_2$}}}%
    \put(0.07547716,0.10203391){\color[rgb]{0,0,0}\makebox(0,0)[lb]{\smash{$N$}}}%
  \end{picture}%
\endgroup%

%% file: thins_nonarithmetic.bbl
\begin{thebibliography}{10}

\bibitem{BLThinBend}
Samuel~A. Ballas and Darren~D. Long.
\newblock Constructing thin subgroups of {${\rm SL}(n + 1, \Bbb{R})$} via
  bending.
\newblock {\em Algebr. Geom. Topol.}, 20(4):2071--2093, 2020.

\bibitem{Bass}
Hyman Bass.
\newblock Groups of integral representation type.
\newblock {\em Pacific J. Math.}, 86(1):15--51, 1980.

\bibitem{BenoistCDIII}
Yves Benoist.
\newblock Convexes divisibles. {III}.
\newblock {\em Ann. Sci. \'Ecole Norm. Sup. (4)}, 38(5):793--832, 2005.

\bibitem{Bergeron}
Nicolas Bergeron.
\newblock Premier nombre de {B}etti et spectre du laplacien de certaines
  vari\'{e}t\'{e}s hyperboliques.
\newblock {\em Enseign. Math. (2)}, 46(1-2):109--137, 2000.

\bibitem{BourgainExpanderSurvey}
Jean Bourgain.
\newblock Some {D}iophantine applications of the theory of group expansion.
\newblock In {\em Thin groups and superstrong approximation}, volume~61 of {\em
  Math. Sci. Res. Inst. Publ.}, pages 1--22. Cambridge Univ. Press, Cambridge,
  2014.

\bibitem{BourGambSarnAffineSeive}
Jean Bourgain, Alex Gamburd, and Peter Sarnak.
\newblock Affine linear sieve, expanders, and sum-product.
\newblock {\em Invent. Math.}, 179(3):559--644, 2010.

\bibitem{ChenTwinPrime}
Jing~Run Chen.
\newblock On the representation of a large even integer as the sum of a prime
  and the product of at most two primes. {II}.
\newblock {\em Sci. Sinica}, 21(4):421--430, 1978.

\bibitem{CooperFuterThin}
D.~{Cooper} and D.~{Futer}.
\newblock {Ubiquitous quasi-Fuchsian surfaces in cusped hyperbolic
  3-manifolds}.
\newblock {\em ArXiv e-prints}, May 2017.

\bibitem{FuchsRivinThin}
Elena Fuchs and Igor Rivin.
\newblock Generic thinness in finitely generated subgroups of {${\rm SL}_n(\Bbb
  Z)$}.
\newblock {\em Int. Math. Res. Not. IMRN}, (17):5385--5414, 2017.

\bibitem{GPS}
M.~Gromov and I.~Piatetski-Shapiro.
\newblock Nonarithmetic groups in {L}obachevsky spaces.
\newblock {\em Inst. Hautes \'{E}tudes Sci. Publ. Math.}, (66):93--103, 1988.

\bibitem{JohnMill}
Dennis Johnson and John~J. Millson.
\newblock Deformation spaces associated to compact hyperbolic manifolds.
\newblock In {\em Discrete groups in geometry and analysis ({N}ew {H}aven,
  {C}onn., 1984)}, Progr. Math.

\bibitem{KahnLabourieMozes}
Jeremy Kahn, Francois Labourie, and Shahar Mozes.
\newblock Surface subgroups in uniform lattices of some semi-simple lie groups.
\newblock {\em in preparation}, 2018.

\bibitem{KahnMark}
Jeremy Kahn and Vladimir Markovic.
\newblock Immersing almost geodesic surfaces in a closed hyperbolic three
  manifold.
\newblock {\em Ann. of Math. (2)}, 175(3):1127--1190, 2012.

\bibitem{KahnWright}
Jeremy {Kahn} and Alex {Wright}.
\newblock {Nearly Fuchsian surface subgroups of finite covolume Kleinian
  groups}.
\newblock {\em arXiv e-prints}, page arXiv:1809.07211, Sep 2018.

\bibitem{LubotExpanderSurvey}
Alexander {Lubotzky}.
\newblock {Expander Graphs in Pure and Applied Mathematics}.
\newblock {\em arXiv e-prints}, page arXiv:1105.2389, May 2011.

\bibitem{MarquisBending}
Ludovic Marquis.
\newblock Exemples de vari\'et\'es projectives strictement convexes de volume
  fini en dimension quelconque.
\newblock {\em Enseign. Math. (2)}, 58(1-2):3--47, 2012.

\bibitem{MRSCollisions}
D.~B. McReynolds, Alan~W. Reid, and Matthew Stover.
\newblock Collisions at infinity in hyperbolic manifolds.
\newblock {\em Math. Proc. Cambridge Philos. Soc.}, 155(3):459--463, 2013.

\bibitem{MilaThesis}
Olivier Mila.
\newblock {\em The trace field of hyperbolic gluings}.
\newblock 2019.
\newblock Thesis (Ph.D.)--Universit\"at Bern.

\bibitem{Rat}
John~G. Ratcliffe.
\newblock {\em Foundations of hyperbolic manifolds}, volume 149 of {\em
  Graduate Texts in Mathematics}.
\newblock Springer, New York, second edition, 2006.

\bibitem{VinbergFieldofDef}
\`E.~B. Vinberg.
\newblock Rings of definition of dense subgroups of semisimple linear groups.
\newblock {\em Izv. Akad. Nauk SSSR Ser. Mat.}, 35:45--55, 1971.

\bibitem{WitteMorris}
D.~{Witte Morris}.
\newblock {Introduction to Arithmetic Groups}.
\newblock {\em ArXiv Mathematics e-prints}, June 2001.

\end{thebibliography}
